\documentclass[preprint,review,10pt]{amsart}
\usepackage{amsmath,amssymb,amsfonts,xspace}
\newtheorem{theorem}{Theorem}[section]

\theoremstyle{definition}
\newtheorem{definition}[theorem]{Definition}
\newtheorem{example}[theorem]{Example}

\newtheorem{corollary}[theorem]{Corollary}
\newtheorem{lem}[theorem]{Lemma}

\theoremstyle{remark}

\numberwithin{equation}{section}

\begin{document}

\title{ $J$-prime hyperideals and their generalizations  }

\author{M. Anbarloei}
\address{Department of Mathematics, Faculty of Sciences,
Imam Khomeini International University, Qazvin, Iran.
}

\email{m.anbarloei@sci.ikiu.ac.ir }

%\author{}
%\address{}
%\email{}

%    General info
\subjclass[2010]{ 20N20, 16Y99}

%\date{September  , 2013.}

\keywords{ $J$-prime hyperideal, Quasi $J$-prime hyperideal, 2-absorbing $J$-prime hyperideal.}
%------------------------------------------------------------------------------
%%%%%%%%%%%%%%%%%%%%%%%%%%%%%%%%%%%%%%%%%%%%%%%%%%%%%%%%%%%%%%%%%%%%%%%%%%%%%%%%%%%%%%%%%%%%%%%%%%%%%%

%%%%%%%%%%%%%%%%%%%%%%%%%%%%%%%%%%%%%%%%%%%%%%%%%%%%%%%%%%%%%%%%%%%%%%%%%%%%%%%%%%%%%%%%%%%%%%%%%%%%%%%%%%%%%%%%%%%%%%%%%%%%%%%%%%%%%%%%%
\begin{abstract}
Let $R$ be a multiplicative hyperring with identity. In this paper,   we define  the concept of $J$-prime hyperideals which is a generalization of $n$-hyperideals and we will show some properties of them. Then, we extend the notion of $J$-prime to quasi $J$-prime and 2-absorbing $J$-prime hyperideals.  
  Various  characterizations
of them are provided. 
\end{abstract}
%%%%%%%%%%%%%%%%%%%%%%%%%%%%%%%%%%%%%%%%%%%%%%%%%%%%%%%%%%%%%%%%%%%%%%%%%%%%%%%%%
\maketitle
\section{Introduction}
The hyperstructure theory was introduced by Marty in 1934, at the 8th Congress of Scandinavian Mathematicians
\cite{sorc1}, when he defined the hypergroups and began to investigate their properties with applications to groups,  algebraic functions and rational fractions. Later on,
many researchers have worked on this new field of modern algebra and developed it. 
The multiplicative hyperring, as one important class of hyperrings, was introduced by Rota in 1982 \cite{rota} and was subsequently investigated by many authors \cite{hq1, hq2, hq3, hq4}.  In the hyperring, the multiplication is a hyperoperation, while the addition is an operation. In 1990, the strongly distributive multiplicative hyperrings were characterized by Rota \cite{hq5}. The polynomials over multiplicative hyperrings were studied by Procesi and Rota in \cite{poly1}. Ameri and Kordi introduced the notions of clean multiplicative hyperring and regular multiplicative hyperring, as two generalizations of classical rings, in \cite{amer66} and \cite{amer3}. 
The concept of derivation on multiplicative hyperrings was introduced by Ardekani and Davvaz in \cite{hq21}. Ameri and et. al.  \cite{amer2} have studied notion of hyperring of fractions generated by a multiplicative hyperring. Soltani and et. al. introduced  zero-divisor graphs of a commutative multiplicative hyperring, as a generalization of commutative rings \cite{graph}.
The codes over multiplicative hyperrings were studied by Akbiyik \cite{akbiyik}.

In the theory of rings, the key role of the notion of prime ideal as a generalization of the notion of prime number in the ring $\mathbb {Z}$ is  undeniable. The notion of primeness of hyperideal in a multiplicative hyperring was conceptualized by Procesi and Rota in \cite{hq4}.
The notions of prime and primary hyperideals in multiplicative hyperrings were fully studied by Dasgupta in \cite{ref1}. The concept of $S$-prime hyperideals in multiplicative hyperrings which is a generalization of prime hyperideals was studied in \cite{s-prime}.
Badawi \cite{hq17} introduced and studied a generalization of prime ideals called
2-absorbing ideals and this notion is further generalized by Anderson and Badawi \cite{hq18, hq19}. 
In \cite{hq13}, Ghiasvand introduced the concept of 2-absorbing
hyperideal in a multiplicative hyperring which is a generalisation of prime hyperideals. Several authors have extended and generalized this concept in several ways \cite{mah, hq14, hq15, hq16}. In \cite{ul}, Ulucak  defined the notion of $\delta$-primary
hyperideals in multiplicative hyperrings, which unifies the prime and primary hyperideals under one frame. Recently, we introduce the notions of $n$-hyperideals and $r$-hyperideal in a multiplicative hyperring \cite{n}.

Our aim in this paper is to  introduce and study the concept of $J$-prime hyperideals which is a generalization of $n$-hyperideals. Furthermore, we defined two generalizations of them in a multiplicative hyperring.  The paper is orgnized as follows. In Section 2, we have given some basic definitions and results of multiplicative hyperrings which we need to develop our paper. 
In Section 3, we  introduce the concept of $J$-prime hyperideals and discuss their relations with some other types of hyperideals.   Furthermore, 
we investigate the behaviour of $J$-prime hyperideals under a good homomorphism. In Section 4, we study  a generalization of the  $J$-prime hyperideals which is called  quasi $J$-prime hyperideals .  We give a characterization of local multiplicative hyperrings in terms of  quasi $J$-prime hyperideals. In Section 5, we extend the notion of $J$-prime to 2-absorbing $J$-prime hyperideals and give some properties of them.

\section{Preliminaries}
Recall first the basic terms and definitions from the hyperring theory.  
A hyperoperation on a non-empty set $G$ is a map $"\circ":G \times G \longrightarrow P^*(G)$, where $P^*(G)$ is the set of all the nonempty subsets of $G$. An algebraic system $(G,\circ)$ is called a hypergroupoid. A hypergroupoid $(G,\circ)$ is called a hypergroup if it satisfies the
following:

(1) $a \circ (b \circ c)=(a \circ b ) \circ c$, for all $a,b,c \in G$.

(2) $a \circ G=G \circ a=G$, for all $a \in G$.

A hypergroupoid with the associative hyperoperation is called a semihypergroup.

A non-empty set $R$ with a operation + and a hyperoperation $\circ$ is called a {\it multiplicative hyperring} if it satisfies the
following:

(1)$ (R,+)$ is an abelian group;

(2)$ (R,\circ)$ is a semihypergroup;

(3) for all $a, b, c \in R$, we have $a \circ (b+c) \subseteq a \circ b+a \circ c$ and $(b+c) \circ a \subseteq b \circ a+c \circ a$;

(4) for all $a, b \in R$, we have $a \circ (-b) = (-a) \circ b = -(a \circ b)$.

Let $A$ and $B$ be two nonempty subsets  of $R$ and $r \in R$. Then  we define
\[ A \circ B=\bigcup_{x \in A,\ y \in B}x \circ y, \ \ \ \ A \circ r=A \circ \{r\}\]

A non empty subset $I$ of a multiplicative hyperring $R$ is a {\it hyperideal} if
\begin{itemize}
\item[\rm(i)]~ If $a, b \in I$, then $a - b \in I$;

\item[\rm(iii)]~ If $x \in I $ and $r \in R$, then $r \circ x \subseteq I$.
\end{itemize}
%\end{definition}
%An element $e \in R$ is said to be {\it scalar identity} if %$a=a \circ e$ for all  $a \in R$.

\begin{definition} \cite{davvaz1}
Let $(R_1, +_1, \circ _1)$ and $(R_2, +_2, \circ_2)$ be two multiplicative hyperrings. A mapping from
$R_1$ into $R_2$ is said to be a {\it good homomorphism} if for all $x,y \in R_1$, $\phi(x +_1 y) =\phi(x)+_2 \phi(y)$ and $\phi(x \circ_1 y) = \phi(x)\circ_2 \phi(y)$.
\end{definition}
\begin{definition} \cite{ref1}
A  proper hyperideal $P$ of $R$ is called a {\it prime hyperideal} if $x\circ y \subseteq P$ for $x,y \in R$ implies that $x \in P$ or $y \in P$. The intersection of all prime hyperideals of $R$ containing $I$ is called the prime radical of $I$, being denoted by $\sqrt{I}$. If the multiplicative hyperring $R$ does not have any prime hyperideal containing $I$, we define $\sqrt{I}=R$. 
\end{definition}
\begin{definition} \cite{amer2}
A proper hyperideal $I$ of $R$ is {\it maximal} in $R$ if for
any hyperideal $J$ of $R$ with  $I \subseteq  J \subseteq  R$ then $J = R$. Also, we say that $R$ is a local multiplicative hyperring, if it has just one maximal hyperideal. For a hyperring $R$ we deﬁne the Jacobson radical $J(R)$ of $R$ as the intersection of all maximal hyperideals of $R$. \\Moreover, if $I$ is a proper hyperideal of $R$, then the Jacobson radical $J(I)$ is defined as the intersection of all maximal hyperideals of $R$ containing $I$. 
\end{definition}
Let {\bf C} be the class of all finite products of elements of $R$ i.e. ${\bf C} = \{r_1 \circ r_2 \circ  . . .  \circ r_n \ : \ r_i \in R, n \in \mathbb{N}\} \subseteq P^{\ast }(R)$. A hyperideal $I$ of $R$ is said to be a {\bf C}-hyperideal of $R$ if, for any $A \in {\bf C}, A \cap I \neq \varnothing $ implies  $A \subseteq I$.
Let I be a hyperideal of $R$. Then, $D \subseteq \sqrt{I}$ where $D = \{r \in R: r^n \subseteq I \ for \ some \ n \in \mathbb{N}\}$. The equality holds when $I$ is a {\bf C}-hyperideal of $R$(\cite {ref1}, proposition 3.2). In this paper, we assume that all hyperideals are {\bf C}-hyperideal.
%\begin{definition} \cite{ref1}
%A nonzero proper hyperideal $Q$ of $R$ is called a {\it primary hyperideal} if $x\circ y \subseteq Q$ for $x,y \in R$ implies that $x \in Q$ or $y \in \sqrt{Q}$. Since $\sqrt{Q}=P$ is a prime hyperideal of $R$ by Propodition 3.6 in \cite{ref1}, $Q$ is referred to as a P-primary hyperideal of $R$.
%\end{definition}
%\begin{definition} \cite{amer2}
%Let $R$ be a multiplicative hyperring. Then we call %$M_n(R)$ as
%the set of all hypermatixes of $R$. Also, for all $A = %(A_{ij})_{n \times n}, B = (B_{ij})_{n \times n} \in %P^\star (M_n(R)), A \subseteq B$ if and only if $A_{ij} %\subseteq B_{ij}$. 
%\end{definotion}
\begin{definition} 
\cite{amer2} Let $R$ be commutative multiplicative hyperring and $1$ be an identity (i. e., for all $a \in R$, $a \in a\circ 1$). An element $x \in R$ in  is called {\it unit}, if there exists $y \in R$, such that $1 \in x\circ y$. Denote the set of all unit elements in $R$ by
$U(R)$.
\end{definition}
%\begin{definition} \cite{amer}
%An element $x \in R$ is said to be {\it zero divizor}, if there %exists $0 \neq y \in R$ such that $\{0\}=x\circ y$.  The set of %all zero divizor elements in $R$ is denoted by $Z(R)$. 
%\end{definition}
\begin{definition} \cite{amer}
An element $x \in R$ is said to be a zero divizor, if there exists
$0 \neq y \in R$ such that $\{0 \}= x \circ y$.  The set of all zero divizor elements in $R$ is denoted by $Z(R)$.
\end{definition}
%\begin{definition} \cite{hq21}
%A hyperring $R$ is called an {\it integral hyperdomain}, if for all $x, y \in  R$,
%$0 \in  x \circ y$ implies that $x = 0$ or $y = 0$.
%\end{definition}
%\begin{definition} \cite{amer66}
%An element $x$ of $R$ is {\it nilpotent} if there exists an $n$ %such that
%$0 \in x^n$ and denote the set of all nilpotent elements of $R$ %by $N(R)$. 
%\end{definition}
%\begin{definition} 
%A hyperring $R$ is said to be a {\it reduced hyperring} if it has no
%nilpotent elements. That is, if $x^n = 0$ for $x \in R$ and a natural number $n$, then $x = 0$.
%\end{definition}
%\begin{definition}
%A multiplicative hyperring $F$ is called a {\it hyperfield} if  %every non-zero element of $F$ is unit.
%\end{definition}
%\begin{definition} \cite{amer3}
% Let $R$ be a multiplicative hyperring. An  element $r \in %R$  is called {\it regular} if there exists $x \in  R$
%such that $r \in r^2\circ x$. So,  $R$ is called {\it regular %multiplicative hyperring}, if all of elements in $R$ are %regular elements. The set of all regular elements in $R$ is denoted by $V(R)$. 
%\end{definition}
\begin{definition}
Let $I,J$ be two hyperideals of $R$ and $x \in R$. Then define:
\[(I:a)=\{r \in R : r\circ a \subseteq I \}\]
\[ann(x)=\{y \in R : x \circ y =\{0\}\}\]
% \[ann(x) = \{y \in R :   x \circ y=\{0\}\}\]%
\end{definition}
 %A proper ideal $I$ of $R$ is called a 2-absorbing quasi-primary ideal of $R$ if $r(I)$ is a 2-absorbing hyperideal of $R$. It is clear that every 2-absorbing primary ideal of a ring R is a 2-absorbing quasi-primary ideal of $R$
 %%%%%%%%%%%%%%%%%%%%%%%%%%%%%%
%%%%%%%%%%%%%%%%%%%%%%%%%%%%%%%%%%%%%%%%%%
\section{$J$-prime hyperideals }
 \begin{definition} 
Let $I$ be a proper hyperideal of $R$.   We call $I$  a $J$-prime hyperideal  if whenever  $x,y\in R$ with  $x \circ y \subseteq I$ , then either $x \in  J(R)$ or $y \in  I$. 
\end{definition}
\begin{example} \label{exa1}
Let $\mathbb{Z}$ be the ring of integers. For all $a,b \in \mathbb{Z}$, we define the hyperoperation $a \circ b =\{a.x.b \ \vert \ x \in A\}$ where $A=\{2,3\}$. Then $(\mathbb{Z},+,\circ)$ 
is a multiplicative hyperring. In the hyperring, every principal hyperideal generated by prime integer is a $J$-prime hyperideal.
\end{example}
\begin{example} In Example \ref{exa1}, let $A=\{14,21\}$. Then the principal hyperideal $\langle 7 \rangle$ is not a $J$-prime hyperideal. Because, $1 \circ 1=\{14,21\}\subseteq \langle 7 \rangle$, but neither $1 \in \langle 7 \rangle$ nor $1 \in J(\mathbb{Z})$.
\end{example}
\begin{theorem} \label{34}
Let $I$ be a $J$-prime hyperideal of $R$. Then $I$ is in $J(R)$.
\end{theorem}
\begin{proof}
Assume that  $I$ is a $J$-prime hyperideal of $R$ such that it is not in $J(R)$. Let $x \in I$ but $x \notin J(R)$. Since $I$ is a {\bf C}-hyperideal and $x \in x \circ 1$, then $x \circ 1 \subseteq I$. Since $I$ is a $J$-prime hyperideal of $R$ and $x \notin J(R)$, then $1 \in I$ which is a contradiction. Thus, $I$ is in $J(R)$.
\end{proof}
\begin{theorem} \label{31}
For a ring R, the following statements are equivalent:
\begin{itemize}
\item[\rm(i)]~ $R$ is a local hyperring.
\item[\rm(ii)]~Every proper hyperideal of  $R$ is   $J$-prime.
\item[\rm(iii)]~Every proper principal hyperideal of $R$ is $J$-prime.
\end{itemize} 
\end{theorem}
\begin{proof}
$(i) \Longrightarrow (ii)$ Suppose that $I$ is a proper hyperideal of $R$ and $M$ is the only maximal hyperideal of $R$. So $J(R)=M$. Let $x \circ y \subseteq I$ for some $x,y \in R$ such that $x \notin M$. Therefore, $x \in U(R)$. Hence, we have $y \in 1 \circ y \subseteq  (x^{-1} \circ x ) \circ y =x^{-1} \circ (x \circ y) \subseteq I$. Thus $I$ is a $J$-prime hyperideal of $R$. 

$(ii) \Longrightarrow (iii)$ Obvious. 

$(iii) \Longrightarrow (i)$ Let every proper principal hyperideal of $R$ is $J$-prime. Suppose that the hyperideal $M$ of $R$ is  maximal. Let $x \in M$.  Since the principal hyperideal $\langle x \rangle$ is $J$-prime and $x \circ 1 \subseteq  \langle x \rangle$, then we get $x \in J(R)$ or $1 \in \langle x \rangle$. In the second case, we have a contradiction. Then  $x \in J(R)$ which means $J(R)=M$. Thus, $R$ is a local hyperring.
\end{proof}
\begin{lem} \label{32}
Each maximal hyperideal $M$ of a multiplicative hyperring $R$ is a prime
hyperideal. 
\end{lem}
\begin{proof}
Let $M$ be a maximal hyperideal of $R$. Suppose that $x \circ y \subseteq M$ for some $x,y \in R$ but $y \notin M$.  So, $\langle y,M \rangle=R$ which means there exists $m \in M$ such that $1 \in \langle y \rangle+m$. Then there exist $r_i \in R$ such that $1 \in \sum_{i=1}^n r_i \circ y +m$. Hence, we get $x \in x \circ 1 \subseteq x \circ (\sum_{i=1}^n r_i \circ y +m) \subseteq (\sum_{i=1}^n x \circ r_i \circ y)+x \circ m \subseteq M$. Then $x \in M$ as needed.
\end{proof}
Recall from \cite{n} that a proper hyperideal $I$ of $R$ is said to be an $n$-hyperideal if $x \circ y \subseteq I$ implies $x \in \sqrt{0}$ or $y \in I$ for any $x,y \in R$.
\begin{theorem}
If $I$ is an n-hyperideal of $R$, then it is $J$-prime. 
\end{theorem}
\begin{proof}
Let $x \circ y \subseteq I$ for some $x,y \in R$ such that $x \notin J(R)$. By Lemma  \ref{32}, we have $\sqrt{0} \subseteq J(R)$. Therefore, $x \notin \sqrt{0}$. Since $I$ is an n-hyperideal of $R$, then $y \in I$. Thus, $I$ is a $J$-prime hyperideal of $R$.
\end{proof}
\begin{theorem}
Let $R$ be a local multiplicative hyperring such that $\sqrt{0} \subsetneqq J(R)$. Then $J(R)$ is a $J$-prime hyperideal of $R$ which is not an n-hyperideal.
\end{theorem}
\begin{proof}
Let $R$ be a local multiplicative hyperring. By Theorem \ref{31}, the hyperideal $J(R)$ of $R$ is $J$-prime.   Let $x \in J(R)-\sqrt{0}$. Then we get  $x \in x \circ 1 \subseteq J(R)$ such that $x \notin \sqrt{0}$ and $1 \notin J(R)$. This implies that $J(R)$ is not an n-hyperideal of $R$.
\end{proof}
Recall from \cite{n} that a proper hyperideal $I$ of $R$ is said to be a $r$-hyperideal if for all  $x,y \in R$, $x \circ y \subseteq I$ and  $ann(x) =\{0\}$, then $y \in I$ .
\begin{theorem}
Let $R$ be a multiplicative hyperring with $Z(R) \subseteq J(R)$. If $I$ is a r-hyperideal of $R$, then $I$ is a $J$-prime hyperideal. 
\end{theorem}
\begin{proof}
Let $I$ be a r-hyperideal of $R$. Suppose that  $x \circ y \subseteq I$ for some $x,y \in R$ such that $x \notin J(R)$. Since $Z(R) \subseteq J(R)$, then $x \notin Z(R)$ which implies $ann(x)=\{0\}$. Since $I$ is a r-hyperideal of $R$, we get $y \in I$. Thus, $I$ is a $J$-prime hyperideal of $R$.
\end{proof}
%\begin{theorem}
%Let $I$ be a $J$-prime hyperideal of $R$. If $I+J=R$ for some hyperideal %J$ of $R$, then $J=R$.
%\end{theorem}
%\begin{proof}

%\end{proof}
\begin{theorem} \label{entersection}
Let $\{I_i\}_{i \in \Delta}$ be a  nonmepty set of $J$-prime hyperideals of $R$. Then $\bigcap_{i \in \Delta}I_i$ is a $J$-prime hyperideal of $R$.
\end{theorem}
\begin{proof}
Let $x \circ y \subseteq \bigcap_{i \in \Delta}I_i$ for some $x,y \in R$ such that $x \notin J(R)$. Then $x \circ y \subseteq I_i$ for every $i \in \Delta$. Since $I_i$ is a $J$-prime hyperideal of $R$, we get the result that $y \in I_i$ and so $y \in \bigcap_{i \in \Delta}I_i$.
\end{proof}
\begin{theorem} \label{36}
Let $I$ be a proper hyperideal of $R$.Then the following statements are equivalent:
\begin{itemize}
\item[\rm(i)]~$I$ is a $J$-prime hyperideal of $R$.
\item[\rm(ii)]~$I=(I:x)$ for every $x \notin J(R)$.
\item[\rm(iii)]~$I_1 \circ I_2 \subseteq I$ for some hyperideals $I_1$ and $I_2$ of $R$ implies that $I_1 \subseteq J(R)$ or $I_2 \subseteq I$.
\end{itemize}
\end{theorem}
\begin{proof}
$(i) \Longrightarrow (ii)$ Let $I$ be a $J$-prime hyperideal of $R$. It is clear that $I \subseteq (I:x)$ for all $x \in R$. Assume that $y \in (I:x)$ such that $x \notin J(R)$. This means $x \circ y \subseteq I$. Since $I$ is a $J$-prime hyperideal of $R$ and $x \notin J(R)$, then $y \in I$. Thus, we have $I=(I:x)$.

$(ii) \Longrightarrow (iii)$ Let $I_1 \circ I_2 \subseteq I$ for some hyperideals $I_1$ and $I_2$ of $R$ such that $I_1 \nsubseteq J(R)$. Therefore, we get $x \in I_1$ such that $x \notin J(R)$. Hence, $x \circ I_2 \subseteq I$ which means $I_2 \subseteq (I:x)$. Since $I=(I:x)$ for every $x \notin J(R)$, then $I_2 \subseteq I$.

$(iii) \Longrightarrow (i)$ Let $x \circ y \subseteq I$ for some $x,y \in R$ such that $x \notin J(R)$. By  Proposition 2.15 in \cite{ref1}, we  have $\langle x \rangle \circ \langle y \rangle \subseteq \langle x \circ y \rangle \subseteq I$ but $\langle x \rangle \nsubseteq J(R)$. Then we get $\langle y \rangle \subseteq I$ which implies $y \in I$. Thus, $I$ is a $J$-prime hyperideal of $R$.
\end{proof}
\begin{theorem}
Let $I$ be a proper hyperideal of $R$. Then  $I$ is a $J$-prime hyperideal of $R$ if and only if $(I:y) \subseteq J(R)$ for every $y \notin I$.
\end{theorem}
\begin{proof}
$\Longrightarrow$ Let $x \in (I:y)$ such that  $y \notin I$. So, $x \circ y \subseteq I$. Since $I$ is a $J$-prime hyperideal of $R$, then $x \in J(R)$.

$\Longleftarrow $ Let $x \circ y \subseteq I$ for some $x,y \in R$ such that $x \notin J(R)$. If $y \notin I$, then $x \in (I:y) \subseteq J(R)$, by the hypothesis. This is a contradiction. Therefore, $y \in I$. Thus, $I$ is a $J$-prime hyperideal of $R$
\end{proof}
\begin{theorem} \label{33}
Let $I$ be a hyperideal of $R$ and let $T$ be  a nonempty subset of $R$ such that $T \nsubseteq I$. If $I$ is a $J$-prime hyperideal of $R$, then $(I:T)$ is a $J$-prime hyperideal of $R$.
\end{theorem}
\begin{proof}
Let $(I:T)=R$. Then $1 \in (I:T)$ which means $T \subseteq I$. This is a contradiction. Hence,  $(I:T)$ is a proper hyperideal of $R$. Suppose that $x \circ y \subseteq (I:T)$ for some $x,y \in R$ such that $x \notin J(R)$. This implies that $x \circ y \circ t \subseteq I$ for all $t \in T$. Then we get $y \circ t \subseteq I$ for all $t \in T$ as $I$ is a $J$-prime hyperideal of $R$. Thus $y \in (I:T)$.
\end{proof}
\begin{theorem} \label{35}
Suppose that  $I$ is a $J$-prime hyperideal of $R$ such that there is no $J$-prime hyperideal which contains $I$ properly. Then $I$ is a prime hyperideal. 
\end{theorem}
\begin{proof}
Suppose that  $I$ is a $J$-prime hyperideal of $R$ such that there is no $J$-prime hyperideal which contains $I$ properly. Let $x \circ y \subseteq I$ for some $x,y \in R$ such that $x \notin I$. By Theorem \ref{33}, $(I:x)$ is a $J$-prime hyperideal of $R$. Since $I \subseteq (I:x)$, we conclude that $y \in (I:x)=I$, by the hypothesis. Thus, $I$ is a prime hyperideal.
\end{proof}
The next Theorem shows that the inverse of Theorem \ref{35} is true if $I=J(R)$.
\begin{theorem}
Let the hyperideal $J(R)$ of $R$ be prime. Then  $J(R)$ is a $J$-prime hyperideal of $R$ such that there is no $J$-prime hyperideal which contains $J(R)$ properly.
\end{theorem}
\begin{proof}
Suppose that  $I=J(R)$. Let $x \circ y \subseteq I$ for some $x,y \in R$ such that $x \notin J(R)$. Since $I$ is a prime hyperideal of $R$, then $y \in I=J(R)$ which means the hyperideal $J(R)$ of $R$ is $J$-prime. By Theorem \ref{34}, we conclude that there is no $J$-prime hyperideal which contains $I$ properly.
\end{proof}
\begin{theorem}
Let $H$ be a hyperideal of $R$ such that $H \nsubseteq J(R)$. Then
\begin{itemize}
\item[\rm(i)]~ If the hyperideals $A_1$ and $A_2$ of $R$ are $J$-prime such that $A_1 \circ H=A_2 \circ H$, then $A_1=A_2$.
\item[\rm(ii)]~ If the hyperideal $I \circ H$ of $R$ is $J$-prime for some hyperideal $I$ of $R$, then $I\circ H=I$.
\end{itemize}
\end{theorem}
\begin{proof}
 $(i)$ It is clear that $A_1 \circ H=A_2 \circ H \subseteq A_2$. By Theorem \ref{36}, we obtain $A_1 \subseteq A_2$ as the hyperideal $A_2$ is $J$-prime. By a similar argument we get $A_2 \subseteq A_1$. Thus $A_1=A_2$. 
 
 $(ii)$ Since $I$ is a hyperideal of $R$, then $I \circ H \subseteq I$. Let the hyperideal $I \circ H$ of $R$ be $J$-prime. Since $I \circ H \subseteq I \circ H$ and  $H \nsubseteq J(R)$, then $I \subseteq I \circ H$, by Theorem \ref{36}. Thus,  $I \circ H=I$.
\end{proof}
\begin{theorem} \label{37}
Let $(R_1,+_1,\circ_1)$ and $(R_2,+_2,\circ_2)$ be multiplicative hyperrings and $\phi: R_1 \longrightarrow R_2$ be a good homomorphism. If $I_2$ is a $J$-prime hyperideal of $R_2$ such that $Ker \phi \subseteq J(R_1)$, then $\phi^{-1}(I_2)$ is a $J$-prime hyperideal of $R_1$.
\end{theorem}
\begin{proof}
Let $I_2$ is a $J$-prime hyperideal of $R_2$. Suppose that $x \circ_1 y \subseteq \phi^{-1}(I_2)$ for some $x,y \in R_1$ such that $x \notin J(R_1)$. This implies that $\phi(x) \circ_2 \phi(y)=\phi(x \circ_1 y) \subseteq I_2$. Let  $M$ be a maximal hyperideal of $R_1$ and  $\phi(x) \in J(R_2)$. Then $\phi(M)$ is a maximal hyperideal of $R_2$ which implies $\phi(x) \in \phi(M)$. Since $Ker \phi \subseteq M$, then we have $x \in M$ which means $x \in J(R_1)$, a contradiction. Therefore $\phi(x) \notin J(R_2)$. Now, we have $\phi(y) \in I_2$ as $I_2$ is a $J$-prime hyperideal of $R_2$. Then we conclude that $y \in \phi^{-1}(I_2)$. Thus, $\phi^{-1}(I_2)$ is a $J$-prime hyperideal of $R_1$.
\end{proof}
\begin{theorem} \label{38}
Let $(R_1,+_1,\circ_1)$ and $(R_2,+_2,\circ_2)$ be multiplicative hyperrings and $\phi: R_1 \longrightarrow R_2$ be a good homomorphism. If $I_1$ is a $J$-prime hyperideal of $R_1$ such that $Ker \phi \subseteq I_1$, then $\phi(I_1)$ is a $J$-prime hyperideal of $R_2$.
\end{theorem}
\begin{proof}
Let $x_2 \circ_2 y_2 \subseteq \phi(I_1)$ for some $x_2,y_2 \in R_2$ such that $x_2 \notin J(R_2)$. Then for some $x_1,y_1 \in R_1$ we have $\phi(x_1)=x_2$ and $\phi(y_1)=y_2$. So $\phi(x_1) \circ_2 \phi(y_1)=\phi(x_1 \circ_1 y_1) \subseteq \phi(I_1)$. Now, take any $u \in ax_1 \circ y_1$. Then $f(u) \in f(x_1 \circ y_1) \subseteq \phi(I_1)$ and so there exists $w \in I_1$ such that $\phi(u)=\phi(w)$. This means  $\phi(u-w)=0$, that is, $u-w \in Ker \phi \subseteq I_1$ and then $u \in I_1$. Since $I_1$ is a {\bf C}-hyperideal of $R_1$, then we get $x_1 \circ_1 y_1 \subseteq I_1$. Since $\phi(J(R_1)) \subseteq J(R_2)$, then $x_1 \notin J(R_1)$. Hence, we have $y_1 \in I_1$ as $I_1$ is a $J$-prime hyperideal of $R_1$. Thus, $y_2=\phi(y_1) \in \phi(I_1)$. It follows that $\phi(I_1)$ is a $J$-prime hyperideal of $R_2$.
\end{proof}
\begin{corollary} \label{39}
Let $H$ and $I$ be two hyperideals of $R$ such that $H \subseteq I$. If $I$ is a $J$-prime hyperideal of $R$, then the hyperideal $I/H$ of $R/H$ is $J$-prime. 
\end{corollary} 
\begin{proof}
Define the natural epimorphism $\pi :R \longrightarrow R/H$ by $\pi(x)=x+H$. Since $Ker \pi  \subseteq I$, we conclude that $\pi(I)=I/H$ is a $J$-prime hyperideal of $R/H$, by Theorem \ref{38}.
\end{proof}
\begin{corollary} \label{310}
Let $H$ and $I$ be two hyperideals of $R$ such that $H \subseteq I \cap J(R)$. If the hyperideal $I/H$ of $R/H$ is $J$-prime, then $I$ is a $J$-prime hyperideal of $R$.
\end{corollary}
\begin{proof}
Consider the natural epimorphism $\pi$ defined in the proof part of Corollary \ref{39}. Now,  the claim follows by  Theorem \ref{37}.
\end{proof}
\begin{corollary} 
Let $H$ and $I$ be two hyperideals of $R$ with $H \subseteq I$ such that $H$ is $J$-prime.  If the hyperideal $I/H$ of $R/H$ is $J$-prime, then $I$ is a $J$-prime hyperideal of $R$.
\end{corollary}
\begin{proof}
This can be proved by using Corollary \ref{310} and Theorem \ref{34}.
\end{proof} 
\begin{definition}
 A nonempty subset $S$ of $R$ containing $R-J(R)$ is called  a $J$-multiplicatively closed subset  if $x \circ y \subseteq S$ for every  $x \in R-J(R)$ and every $y \in S$.
\end{definition}
\begin{theorem}
Let $I$ be a proper hyperideal of $R$. Then $I$ is $J$-prime if and only if $R-I$ is a $J$-multiplicatively closed subset of $R$. 
\end{theorem}
\begin{proof}
$\Longrightarrow$ Let $I$ be a proper hyperideal of $R$. By Theorem \ref{34}, we conclude that  $R-J(R) \subseteq R-I$. Suppose that $x \in R-J(R)$ and $y \in R-I$. To establish the claim, suppose, on the contrary, that $x \circ y \subseteq I$. From $x \in R-J(R)$ it follows that $y \in I$ as $I$ is a $J$-prime of $R$. This is a contradiction. Hence $x \circ y  \subseteq R-I$ as needed. 

$\Longleftarrow$ Let $x \circ y \subseteq I$ for some $x,y \in R$ such that $x \notin J(R)$. If $y \notin I$, then we conclude that $x \circ y \subseteq R-I$.  Thus, we arrive at a contradiction.   Therefore $y \in I$ which means $I$ is $J$-prime.
\end{proof}
\begin{theorem}
Let $S$ be a multiplicatively subset of $R$ and $I$ be a hyperideal of $R$ disjoint from $S$. Then there exists a  hyperideal $Q$ which is maximal in the set of all hyperideals  of $R$ disjoint from $S$, containing $I$. Any such
hyperideal $Q$ is a $J$-prime hyperideal of $R$.
\end{theorem}
\begin{proof}
Let $\Upsilon$ be the set of all hyperideals of $R$ disjoint from $S$, containing $I$. Then $\Upsilon \neq \varnothing$, since $I \in \Upsilon$. So $\Upsilon$ is a partially ordered set with respect to set inclusion relation. By Zorn’s lemma, there is a hyperideal $Q$ which is maximal in $\Upsilon$. Suppose that $Q$ is not  a $J$-prime hyperideal of $R$. Let $x \circ y \subseteq Q$ such that $x \in R-J(R)$ and $y \in R-Q$.  Therefore we have $Q \subsetneqq (Q:x)$. Since $Q$ is a maximal element of $\Upsilon$, then $(Q:x) \cap S \neq \varnothing$. Let $t \in (Q:x) \cap S$. Then $t \circ x \subseteq Q \cap S$ which means $ Q \cap S \neq \varnothing $. Thus we arrive at a contradiction. Consequently, $Q$ is a $J$-prime hyperideal of $R$.
\end{proof}
Let $(R_1,+_1,\circ_1)$ and $(R_2,+_2,\circ_2)$ be two multiplicative hyperrings with non zero identity. \cite{ul} Recall $(R_1 \times R_2, +,\circ)$ is a multiplicative hyperring with the operation $+$ and the hyperoperation $\circ$ are defined respectively as

$(x_1,x_2)+(y_1,y_2)=(x_1+_1y_1,x_2+_2y_2)$ 
and

$(x_1,x_2) \circ (y_1,y_2)=\{(x,y) \in R_1 \times R_2 \ \vert \ x \in x_1 \circ_1 y_1, y \in x_2 \circ_2 y_2\}$. \\
\begin{theorem}
Let $(R_1,+_1,\circ_1)$ and $(R_2,+_2,\circ_2)$ be two multiplicative hyperrings with non zero identity. Then $R_1 \times R_2$ has no $J$-prime hyperideals.
\end{theorem}
\begin{proof}
Let $I_1 \times I_2$ is a $J$-prime hyperideal of $R_1 \times R_2$ for some hyperideals $I_1, I_2$  of $R_1$, $R_2$. Since $(0,0) \in (1,0) \circ (0,1) \cap I_1 \times I_2$, then $(1,0) \circ (0,1) \subseteq I_1 \times I_2$. Let $(1,0) , (0,1) \in J(R_1 \times R_2)$. Then $(1,0)=(1,1)-(0,1)$ and $(0,1)=(1,1)-(1,0)$ are unit elements of $R_1 \times R_2$. Therefore we conclude that $(1,0), (0,1)$ are not in $J(R_1 \times R_2)$. Thus, $(1,0),(0,1)$ are in $I_1 \times I_2$. So, $(1,1)=(1,0)+(0,1) \in I_1 \times I_2$ which means $I_1 \times I_2=R_1 \times R_2$.
\end{proof}
\begin{definition}
A proper hyperideal $I$ of $R$ is called $J$-primary if whenever elements $x,y \in R$ and $x \circ y \subseteq I$, then $x \in J(I)$ or $y \in I$.
\end{definition}
\begin{theorem}
Let $I$ be a hyperideal of $R$ with $I \subseteq J(R)$. Then $I$ is a $J$-prime hyperideal if and only if $I$ is $J$-primary.
\end{theorem}
\begin{proof}
$\Longrightarrow$ Let $I$ be a $J$-prime hyperideal of $R$. Suppose that $x \circ y \subseteq I$ for some $x,y \in R$ such that $x \notin J(I)$. This means $x  \notin J(R)$ as $J(R) \subseteq J(I)$. Since $I$ is a $J$-prime hyperideal of $R$ and $x  \notin J(R)$, we get $y \in I$. Consequently, the hyperideal $I$ is $J$-primary.

$\Longleftarrow$ Suppose that the hyperideal $I$ of $R$ is $J$-primary. Let $x \circ y \subseteq I$ for some $x,y \in R$ such that $x \notin J(R)$. By the hypothesis, we have $J(I) \subseteq J(R)$. This implies that $x \notin J(I)$. Since  $I$  is a $J$-primary  hyperideal of $R$ and $x \notin J(I)$, then we get $y \in I$. Thus, $I$ is a $J$-prime hyperideal.
\end{proof}
%444444444444444444444444444444444444444444444444444444444444
%444444444444444444444444444444444444444444444444444444444444
%4444444444444444444444444444444444444444444444444444444444444
\section{quasi $J$-prime hyperideals}
In this section, we define the concept of quasi $J$-prime hyperideals as a generalization of $J$-prime hyperideals.
\begin{definition}
A proper hyperideal $I$ of $R$ is called a quasi $J$-prime hyperideal if $\sqrt{I}$ is $J$-prime.
\end{definition}
\begin{example} In Example \ref{exa1}, the  hyperideals $\langle 2 \rangle$ and $\langle 3 \rangle$ are quasi $J$-prime. The hyperideal $\langle 12 \rangle$ is not a quasi $J$-prime hyperideal. In fact, $3 \circ 4 =\{24,36\} \subseteq \langle 12 \rangle$, but $3,4 \notin \sqrt{\langle 12 \rangle}=\langle 2 \rangle \cap \langle 3 \rangle$ and $3,4 \notin J(\mathbb{Z})$.
\end{example}
\begin{theorem} \label{41}
Let $I$ be a quasi $J$-prime  hyperideal of $R$. If $x \circ H \subseteq I$ for some hyperideal $H$ of $R$ and some element $x \in R$, then $x \in J(R)$ or $H \subseteq \sqrt{I}$.
\end{theorem}
\begin{proof}
Let $I$ be a quasi $J$-prime  hyperideal of $R$. Suppose that $x \circ H \subseteq I$ for some hyperideal $H$ of $R$ and some element $a \in R$ such that $x \notin J(R)$. By Theorem \ref{36},  we get $\sqrt{I}=(\sqrt{I}:x)$. Since $(I:x) \subseteq (\sqrt{I}:x)$, then $H \subseteq \sqrt{I}$. 
\end{proof}
\begin{theorem} \label{42}
Let $I$ be a proper hyperideal of $R$. Then the following are equivalent:
\begin{itemize}
\item[\rm(i)]~ $I$ is a quasi $J$-prime  hyperideal of $R$.
\item[\rm(ii)]~ If $H \circ T \subseteq I$ for some hyperideals $H$ and $T$ of $R$, then $H \subseteq J(R)$ or $T \subseteq \sqrt{I}$.
\item[\rm(iii)]~ If $x \circ y \subseteq I$ for some $x,y \in R$, then $x \in J(R)$ or $y \in \sqrt{I}$.
\end{itemize}
\end{theorem}
\begin{proof}
$(i) \Longrightarrow (ii)$ Let $H \circ T \subseteq I$ for some hyperideals $H$ and $T$ of $R$ such that $H \nsubseteq J(R)$. Take $x \in H$ such that  $x \notin J(R)$. Clearly, $x \circ T \subseteq I$. Since $I$ is a quasi $J$-prime  hyperideal of $R$ and $x \notin J(R)$, then $T \subseteq \sqrt{I}$, by Theorem \ref{41}. 

$(ii) \Longrightarrow (iii)$ Let $x \circ y \subseteq I$ for some $x,y \in R$. Put $H=\langle x \rangle$ and $T=\langle y \rangle$. Hence, $\langle x \rangle \circ \langle y \rangle \subseteq \langle x \circ y \rangle \subseteq I$. By the assumption, we get $x \in \langle x \rangle \subseteq J(R)$ or $y \in \langle y \rangle \subseteq \sqrt{I}$.

$(iii) \Longrightarrow (i)$ Let $x \circ y \subseteq \sqrt{I}$ for some $x,y \in R$ such that $x \notin J(R)$. This means we have $(x \circ y)^n =x^n \circ y ^n \subseteq I$ for some $n \in \mathbb{N}$. Since $x \notin J(R)$, then $x^n \notin J(R)$, by Theorem \ref{32}. Now, by the assumption, we obtain $y^n \subseteq \sqrt{I}$ which implies $y \in \sqrt{\sqrt{I}}$. Thus $y \in \sqrt{I}$ as needed.
\end{proof}
\begin{theorem} \label{43}
Let $I$ be a proper hyperideal of $R$. Then $I$ is a quasi $J$-prime hyperideal of $R$ if and only if $I \subseteq J(R)$ and for $x,y \in R$, $x \circ y \subseteq I$ implies that $x \in J(I)$ or $y \in \sqrt{I}$.
\end{theorem}
\begin{proof}
$\Longrightarrow $ Let $\sqrt{I}$ be a  $J$-prime hyperideal of $R$. By Theorem \ref{34}, we have $ \sqrt{I} \subseteq J(R)$. Since $I \subseteq \sqrt{I}$, then we get $I \subseteq J(R)$. Let $x \circ y \subseteq I$ for some $x,y \in R$. Then we have $x \circ y \subseteq \sqrt{I}$. Since $\sqrt{I} $ is a $J$-prime hyperideal of $R$, then we obtain $x \in J(R)\subseteq J(I)$ or $y \in \sqrt{I}$. 

$\Longleftarrow$ Let $x \circ y \subseteq I$ such that $x \notin J(R)$. By the hypothesis, we get $J(I) \subseteq J(R)$ which implies $x \notin J(I)$. Hence, we have $y \in \sqrt{I}$. We conclude that the hyperideal $I$ of $R$ is a quasi $J$-prime hyperideal. 
\end{proof}
%\begin{theorem}
%Let $R$ be a multilicative hyperring such that $J(R)=0$. Then $R$ is an integral %hyperdomain if and only if $0$ is the only quasi $J$-prime hyperideal of $R$.
%\end{theorem}
%\begin{proof}
%\end{proof}
\begin{lem} \label{44}
Let $I$ be a hyperideal of $R$ and $S$ be a subset of $R$ with $S \nsubseteq J(R)$. If $I$ is a quasi $J$-prime hyperideal of $R$, then $(\sqrt{I}:S) \subseteq \sqrt{(I:S)}$.
\end{lem}
\begin{proof}
Let $x \in (\sqrt{I}:S)$. Then $x \circ S=\sqrt{I}$. Since $x \circ S=\bigcup_{y \in S}x \circ  y $,  we have $x \circ y \subseteq \sqrt{I}$ for every $y \in S$. Since $S \nsubseteq J(R)$, then there exists $z \in S-J(R)$ such that $x \circ z \subseteq \sqrt{I}$. Since $\sqrt{I}$ is a $J$-prime hyperideal of $R$, we get $x \in \sqrt{I}$. Therefore for some $n \in \mathbb{N}$, $x^n \subseteq I$. This implies that $x^n \subseteq (I:S)$ which means $x \in \sqrt{(I:S)}$. Consequently, $(\sqrt{I}:S) \subseteq \sqrt{(I:S)}$.
\end{proof}
\begin{theorem} \label{45}
Let $I$ be a hyperideal of $R$ and $S$ be a subset of $R$ such that  $S \nsubseteq J(R)$. If the hyperideal $I$ of $R$ is quasi $J$-prime, then so is $(I:S)$.  
\end{theorem}
\begin{proof}
It is clear that $(I:S)$ is a proper hyperideal of $R$. Let $x \circ y \subseteq (I:S)$ for some $x,y \in R$ such that $x \notin J(R)$. Then $x \circ y \circ S \subseteq I$. 
Now we have $x \circ r \subseteq I$ for all $r \in y \circ S$. By Theorem \ref{42}, we get $r \in \sqrt{I}$. Since $y \circ S \cap \sqrt{I} \neq \varnothing$ and $I$ is a {\bf C}-hyperideal, then we conclude that $y \circ S \subseteq \sqrt{I}$. This means $y \in (\sqrt{I}:S)$ which implies $y \in \sqrt{(I:S)}$, by Lemma \ref{44}. Thus, $(I:S)$ is a quasi $J$-prime hyperideal of $R$.
\end{proof}
\begin{theorem}
Suppose that  $I$ is a quasi $J$-prime hyperideal of $R$ such that there is no quasi
$J$-prime hyperideal  which contains $I$ properly. Then the hyperideal $I$ of $R$ is  $J$-prime.
\end{theorem}
\begin{proof}
Suppose that  $I$ is a quasi $J$-prime hyperideal of $R$ such that there is no quasi
$J$-prime hyperideal  which contains $I$ properly. Let $x \circ y \subseteq I$ for some $x,y \in R$ such that $x \notin J(R)$. By Theorem \ref{45}, we conclude that $(I:x)$ is a quasi $J$-prime hyperideal of $R$. Since $I \subseteq (I:x)$, then $I=(I:x)$, by the hypothesis. Thus, we get $y \in I$. Consequently, the hyperideal $I$ of $R$ is  $J$-prime.
\end{proof}
\begin{corollary}
The hyperideal $J(R)$ of $R$ is $J$-prime if and only if $J(R)$ is quasi $J$-prime.
\end{corollary}
\begin{theorem} \label{0913}
Let $M$ be a maximal hyperideal of $R$. If $\langle a \circ b \rangle$ is  a quasi $J$-prime hyperideal of $R$  for each  $a,b \in R$ is  quasi $J$-prime, then so is $M$.  
\end{theorem}
\begin{proof}
Let $M$ be a maximal hyperideal of $R$. Let $x \circ y \subseteq M$ and $x \notin \sqrt{M}$. Since $\sqrt{\langle x \circ y  \rangle } \subseteq \sqrt{M}$ and $ x \notin \sqrt{M}$, then  $x \notin \sqrt{\langle x \circ y  \rangle}$. Since  $\langle x \circ y  \rangle$ is  a quasi $J$-prime hyperideal of $R$ and $x \circ y \subseteq \langle x \circ y  \rangle$, then we obtain  $y \in J(R)$ which means $M$ is a quasi $J$-prime hyperideal of $R$. 
\end{proof}
\begin{theorem} \label{0914}
Every maximal hyperideal  of $R$ is a quasi $J$-prime hyperideal if and only if  $R$ is a local hyperring.
\end{theorem}
\begin{proof}
$\Longrightarrow$ Let the maximal hyperideal $M$ of $R$ be a quasi $J$-prime hyperideal. By Theorem \ref{34}, we get $M\subseteq J(R)$. Since $M$ is maximal, then $M$ is a prime hyperideal of $R$. Hence $M=\sqrt{M}$ and so $M \subseteq J(R)$. 

$\Longleftarrow$ Let $R$ ibe a local hyperring. By Theorem \ref{31}, every proper hyperideal of $R$ is $J$-prime and so every proper hyperideal of $R$ is quasi $J$-prime. Now, the claim follows by Theorem \ref{0913}.
\end{proof}
Recall from \cite{m7m} that a proper hyperideal $I$ of $R$ is said to be a quasi primary hyperideal if $\sqrt{I}$ is prime.
 
\begin{theorem} \label{0912}
Let every prime hyperideal of $R$ be maximal. If the hyperideal $I$  of $R$ is quasi $J$-prime and $I \subseteq J(R)$, then $I$ is a quasi primary hyperideal of $R$.
\end{theorem}
\begin{proof}
We show that $\sqrt{I}$ is prime. Let $x \circ y \subseteq \sqrt{I}$ for some $x,y \in R$ such that $x \notin \sqrt{I}$. Hence we have $x^n \circ y^n \subseteq I$ for some $n \in \mathbb{N}$. By the hypothesis, we conclude that $\sqrt{I}=J(R)$. Since $x^n \notin J(R)$ and the hyperideal $I$ of $R$ is quasi $J$-prime, then $y^n \subseteq \sqrt{I}$, by Theorem \ref{42}. This means $y \in \sqrt{I}$. Thus $I$ is a quasi primary hyperideal of $R$.
\end{proof}
\begin{theorem}
Let every prime hyperideal of $R$ be maximal and $I$ be a hyperideal of $R$ such that $I \subseteq J(R)$. Then $I$   is quasi $J$-prime if and only if  the multiplicative hyperring $R$ is local with the maximal hyperideal $\sqrt{I}$.
\end{theorem}
\begin{proof}
$\Longrightarrow $ Let $I$ be a quasi $J$-prime hyperideal of $R$. Hence $\sqrt{I}$ is a prime hyperideal of $R$, by Theorem \ref{0912}. By the hypothesis, $\sqrt{I}$ is a maximal hyperideal of $R$. Then there exists some prime hyperideal $Q$ such that $I=Q^n$ for some $n \in \mathbb{N}$. This means $\sqrt{I}=Q$ is maximal and so $J(R) \subseteq \sqrt{I}=Q$. Since $I \subseteq J(R)$, then we conclude that $\sqrt{I}=J(R)$. Thus the multiplicative hyperring $R$ is local with the maximal hyperideal $\sqrt{I}$.

$\Longleftarrow$ The claim follows by Theorem \ref{0914}.
\end{proof}
\begin{theorem}
\begin{itemize}
\item[\rm(i)]~If $I_i$ is a quasi $J$-prime hyperideal of $R$ for each $1 \leq i \leq n$, then is so $\bigcap_{i=1}^nI_i$. 
\item[\rm(ii)]~If $I_i$ is a quasi $J$-prime hyperideal of $R$ for each $1 \leq i \leq n$, then is so $\Pi_{i=1}^nI_i$.
\end{itemize}
\end{theorem}
\begin{proof}
(i) By Proposition 3.3. in \cite{ref1}, we have $\sqrt{\bigcap_{i=1}^nI_i}=\bigcap_{i=1}^n \sqrt{I_i}$. Since $\sqrt{I_i}$ is a  $J$-prime hyperideal of $R$ for each $1 \leq i \leq n$, we conclude that $\bigcap_{i=1}^n \sqrt{I_i}$ is a  $J$-prime hyperideal of $R$, by Theorem \ref{entersection}. Thus, $\sqrt{\bigcap_{i=1}^nI_i}$ is a  $J$-prime hyperideal of $R$ which means $\bigcap_{i=1}^nI_i$ is a quasi $J$-prime hyperideal of $R$.

(ii) Let $x \circ y \subseteq \Pi_{i=1}^nI_i$ for some $x,y \in R$ such that $x \notin J(R)$. Since $I_i$ is a quasi $J$-prime hyperideal of $R$ for each $1 \leq i \leq n$ and $\Pi_{i=1}^nI_i \subseteq \bigcap_{i=1}^nI_i$, we conclude that $y \in \sqrt{I_i}$ for each $1 \leq i \leq n$. This means $y^{t_i} \subseteq I_i$ for each $1 \leq i \leq n$. Put $t=t_1+t_2+...+t_n$. Therefore, $y^{t}=y^{t_1} \circ y^{t_2} \circ ... \circ y^{t_n} \subseteq \Pi_{i=1}^nI_i$ which implies $y \in \sqrt{\Pi_{i=1}^nI_i}$. Consequently, the hyperideal $\Pi_{i=1}^nI_i$ of $R$ is quasi $J$-prime.
\end{proof}
%6666666666666666666666666666666666666666666666666666
%66666666666666666666666666666666666666666666666666666
%66666666666666666666666666666666666666666666666666666
\section{2-absorbing $J$-prime hyperideals}
In this section , we extend the notion of $J$-prime to 2-absorbing $J$-prime hyperideals and give some properties of them.
\begin{definition}
Let $I$ be a proper hyperideal of $R$. $I$ is called a 2-absorbing $J$-prime hyperideal of $R$ if whenever $x,y,z \in R$ with $x \circ y \circ z \subseteq I$, then $x \circ y \subseteq I$ or $x \circ z \subseteq J(R)$ or $y \circ z \subseteq J(R)$.
\end{definition}
\begin{example} In Example \ref{exa1}, let $A=\{2,4\}$. Then the principal hyperideal $\langle 15 \rangle$ is  a 2-absorbing $J$-prime hyperideal. 
\end{example}
\begin{example}
Consider the ring $(\{\bar{0},\bar{1},\bar{2},\bar{4},\bar{5}\}=\mathbb{Z}_6,\oplus,\star)$ that for all $\bar{x}, \bar{y}\in \mathbb{Z}_6$, $\bar{x} \oplus \bar{y} $ and $\bar{x} \star \bar{y}$ are the remainder of $\frac{x+y}{6}$ and $\frac{x \cdot y}{6}$, respectively, which $"+"$ and $"\cdot"$ are
ordinary addition and multiplication, and $x,y \in \mathbb{Z}$. We define the hyperoperation $\bar{x} \odot \bar{y}=\{\bar{xy},\bar{2xy},\bar{3xy},\bar{4xy},\bar{5xy}\}$, for all $\bar{x},\bar{y} \in\mathbb{Z}_6$. In the commutative multiplicative hyperring $(\mathbb{Z},\oplus,\odot)$, hyperideal $\{\bar{0}\}$ is 2-absorbing $J$-prime.
\end{example}
\begin{theorem} \label{61}
If $I$ is a $J$-prime hyperideal of $R$, then $I$ is a 2-absorbing $J$-prime hyperideal of $R$.
\end{theorem}
\begin{proof}
Let $I$ be a $J$-prime hyperideal of $R$. Suppose that $x \circ y \circ z \subseteq I$ for some $x,y,z \in R$. Choose  $u \in x \circ z$. Since $y \circ u \subseteq I$ and $I$ is a $J$-prime hyperideal of $R$, we get $y \in I$ or $u \in J(R)$. In the former case, we have $y \circ z \subseteq I$ which means $y \circ z \subseteq J(R)$, by Theorem \ref{34}. In the second case, we obtain $x \circ z \subseteq J(R)$ as $J(R)$ is a {\bf C}-hyperideal of $R$.
\end{proof}
\begin{theorem} \label{62}
If $I$ is a 2-absorbing $J$-prime hyperideal of $R$, then $I \subseteq J(I)$.
\end{theorem}
\begin{proof}
Let $I$ be a 2-absorbing $J$-prime hyperideal of $R$. We suppose that $I \nsubseteq J(I)$ and look for a contradiction. From $I \nsubseteq J(I)$ it follows that there exists $x \in I$ such that $x \notin J(R)$. Since  $1 \circ 1 \circ x \subseteq I$ and  $I$ is a 2-absorbing $J$-prime hyperideal of $R$, we have $1 \circ 1 \subseteq I$, a contradiction or $x \in 1 \circ x \subseteq J(R)$, a contradiction. Thus,  $I \subseteq J(I)$.
\end{proof}
Recall from \cite{mah} that a proper hyperideal $I$ of $R$ is said to be an 2-absorbing primary hyperideal if $x \circ y \circ z \subseteq I$ implies $x \circ y \subseteq I$ or $x \circ z  \subseteq \sqrt{I}$ or $y \circ z \subseteq \sqrt{I}$ for any $x,y,z \in R$.\\
The next Theorem shows that the inverse of Theorem \ref{62} is true if $I$ is a 2-absorbing primary hyperideal of $R$
\begin{theorem} \label{63}
Let $I$ be a 2-absorbing primary hyperideal of $R$ and  $I \subseteq J(R)$. Then $I$ is a 2-absorbing $J$-prime hyperideal of $R$. 
\end{theorem}
\begin{proof}
Suppose that $I$ is a 2-absorbing primary hyperideal of $R$ such that  $I \subseteq J(R)$. Let $x\circ y \circ z \subseteq I$ for some $x,y,z \in R$ such that $x \circ z, y \circ z \nsubseteq J(R)$. Then we conclude that $x \circ z, y \circ z \nsubseteq \sqrt{I}$. This implies that $x \circ y \subseteq I$ as $I$ is a 2-absorbing primary hyperideal of $R$. Consequently, $I$ is a 2-absorbing $J$-prime hyperideal of $R$. 
\end{proof}
\begin{theorem} \label{64}
Let the multiplicative hyperring $R$ has at most two maximal hyperideals. Suppose that the hyperideal $I$ of $R$ is 2-absorbing primary that is not quasi primary. Then $I$ is a 2-absorbing $J$-prime hyperideal of $R$.
\end{theorem}
\begin{proof}
Assume that the multiplicative hyperring $R$ has at most two maximal hyperideals. Let the hyperideal $I$ of $R$ be 2-absorbing primary. Then $\sqrt{I}=P$ or $\sqrt{I}=P_1 \cap P_2$ for some prime hyperideals of $R$, by Theorem 4.5 in \cite{mah}. In the former case, we have a contradiction since $I$ is not a quasi primary hyperideal of $R$. Therefore, we obtain $I \subseteq \sqrt{I} \subseteq J(R)$ as the multiplicative hyperring $R$ has at most two maximal hyperideals. Thus, $I$ is a 2-absorbing $J$-prime hyperideal of $R$, by Theorem \ref{63}.
\end{proof}
\begin{corollary} \label{65}
Let $R_1$ and $R_2$ be two local multiplicative hyperrings. Then every 2-absorbing primary hyperideal of $R_1 \times R_2$ that is not quasi primary is 2-absorbing $J$-prime.
\end{corollary}
\begin{theorem} \label{66}
Let $I$ be a 2-absorbing $J$-prime of $R$. Then for each $a,b \in I$ with $a \circ b \nsubseteq I$, $(I:a \circ b) \subseteq (J(R):a)$ or $(I:a \circ b) \subseteq (J(R):b)$.
\end{theorem}
\begin{proof}
Let $I$ be a 2-absorbing $J$-prime of $R$. Let $c \in (I:a \circ b)$ for some $a,b \in R$ with $a \circ b \nsubseteq I$. This means $a \circ b \circ c \subseteq I$. Since the hyperideal $I$ of $R$ is 2-absorbing $J$-prime and $a \circ b \nsubseteq I$, then $a \circ c \subseteq J(R)$ of $b \circ c \subseteq J(R)$ which means $c \in (J(R):a)$ or $c \in (J(R):b)$. This implies that $(I:a \circ b) \subseteq (J(R):a)$ or $(I:a \circ b) \subseteq (J(R):b)$.
\end{proof}
Recall that a hyperideal $I$ of $R$ is called a strong {\bf C}-hyperideal if for any $E \in \mathfrak{U}$, $E \cap I \neq \varnothing$, then $E \subseteq I$, where $\mathfrak{U}=\{\sum_{i=1}^n A_i \ : \ A_i \in {\bf C}, n \in \mathbb{N}\}$ and ${\bf C} = \{r_1 \circ r_2 \circ  . . .  \circ r_n \ : \ r_i \in R, n \in \mathbb{N}\}$.
\begin{theorem}  \label{67}
Let $I$ be a proper strong {\bf C}-hyperideal of $R$. Then the following
 are equivalent:
 \begin{itemize}
\item[\rm(i)]~$I$ is a 2-absorbing $J$-prime hyperideal of $R$.
\item[\rm(ii)]~ If $a \circ b \circ H \subseteq I$ for some $a,b \in R$ and some hyperideal $H$ of $R$, then $a \circ b \subseteq I$ or $a \circ H \subseteq J(R)$ or $b \circ H \subseteq J(R)$.
\item[\rm(iii)]~If $a \circ H \circ T \subseteq I$ for some $a \in R$ and some hyperideals $H, T$ of $R$, then $a \circ H \subseteq I$ or $a \circ T \subseteq J(R)$ or $H \circ T \subseteq J(R)$.
\item[\rm(iv)]~If $K \circ H \circ T \subseteq I$ for  some hyperideals $K, H, T$ of $R$, then $K \circ H \subseteq I$ or $K \circ T \subseteq J(R)$ or $H \circ T \subseteq J(R)$.
\end{itemize}
\end{theorem}
\begin{proof}
$(i)\Longrightarrow (ii)$ Let $a \circ b \circ H \subseteq I$ for some $a,b \in R$ and some hyperideal $H$ of $R$ such that $a \circ b \nsubseteq I$. Hence, we get $H\subseteq (J(R):a)$ or $H \subseteq (J(R):b)$, by Theorem \ref{66}. This implies that $a \circ H \subseteq J(R)$ or $b \circ H \subseteq J(R)$.

$(ii)\Longrightarrow (iii)$ Let $a \circ H \circ T \subseteq I$ for some $a \in R$ and some hyperideals $H, T$ of $R$ such that $a \circ H \nsubseteq I$ and $a \circ T \nsubseteq J(R)$. Therefore, we get $a \circ h_1 \nsubseteq I$  and $a \circ t_1 \nsubseteq J(R)$ for some $h_1 \in H$ and $t_1 \in T$. Since $a \circ h_1 \circ T \subseteq I$, then $h_1 \circ T \subseteq J(R)$, by (i). Let $h_2 \in H$. Suppose that $h_2 \circ T \nsubseteq J(R)$. By (i), we get $a \circ h_2 \subseteq I$ as $a \circ h_2 \circ T  \subseteq I$. Hence $a \circ (h_1+h_2) \nsubseteq I$. Again by (i), we obtain $(h_1+h_2) \circ T \subseteq J(R)$ as $a \circ (h_1+h_2) \circ T \subseteq I$. As $h_1 \circ T \subseteq J(R)$, $h_2 \circ T \subseteq J(R)$ which means $H \circ T \subseteq J(R)$.

$(iii)\Longrightarrow (iv)$ Let $K \circ H \circ T \subseteq I$ for  some hyperideals $K, H, T$ of $R$ such that $K \circ H \nsubseteq I$ and $H \circ T \nsubseteq J(R)$. Then we have $k \circ H \nsubseteq I$ for some $k \in  K$. It is clear that $k \circ H \circ T \subseteq I$. Hence we have $k \circ T \subseteq J(R)$, by (iii). Assume that $x \in K$. By (iii),  we conclude that $x \circ H \subseteq I$ or $x \circ T \subseteq J(R)$ as $x \circ H \circ T \subseteq I$. If $x \circ H \subseteq I$, then $(k+x) \circ H \subseteq (k \circ H)+(x \circ H) \nsubseteq I$. Clearly, $(k+x) \circ H \circ T \subseteq I$. Then $(k+x) \circ T \subseteq J(R)$. Since $I$  is a strong {\bf C}-hyperideal of $R$, then $(k \circ T)+(x \circ T) \subseteq J(R)$. Hence, $x \circ T \subseteq J(R)$ as  $k \circ T \subseteq J(R)$. Now, let $x \circ H \nsubseteq I$. It is obvious that $x \circ H \circ T \subseteq I$. Therefore, $x \circ T \subseteq J(R)$, by (iii). Consequently, $K \circ T \subseteq J(R)$.

$(iv)\Longrightarrow (i)$ Let $x \circ y \circ z \subseteq I$ for some $x,y,z \in R$ such that $x \circ z \nsubseteq J(R)$ and $y \circ z \nsubseteq J(R)$. We consider $K=\langle x \rangle$, $H=\langle y \rangle$ and $T=\langle z \rangle$. Thus, we conclude that $K \circ T=\langle x \rangle \circ \langle z \rangle \subseteq \langle x \circ z \rangle \nsubseteq J(R)$ and $H \circ T=\langle y \rangle \circ \langle z \rangle \subseteq \langle y \circ z \rangle \nsubseteq J(R)$. Since $K \circ H \circ T \subseteq I$, by (iv), we get $K \circ H  \subseteq I$ which means $x \circ y \subseteq I$ and that completes the proof.
\end{proof}

\begin{theorem}
If every proper hyperideal of $R$ is 2-absorbing $J$-prime, then $R$ is a local multiplicative hyperring.
\end{theorem}
\begin{proof}
Let every proper hyperideal of $R$ be 2-absorbing $J$-prime. Suppose that $M$ is a maximal hyperideal of $R$. We show that $M  \subseteq J(R)$. Let $a \in M$ and $I=\langle a \rangle$. Since $I$ is a 2-absorbing $J$-prime of $R$, then $I \subseteq J(R)$ by Theorem \ref{62}. Therefore, $a \in J(R)$ which means $M \subseteq J(R)$. Thus $J(R) =M$ as $J(R) \subseteq M$. Consequently, $R$ is a local multiplicative hyperring.
\end{proof}
\begin{theorem} \label{entersection2}
Let $\{I_i\}_{i \in \Delta}$ be a  nonmepty set of 2-absorbing $J$-prime hyperideals of $R$. Then $\bigcap_{i \in \Delta}I_i$ is a 2-absorbing $J$-prime hyperideal of $R$.
\end{theorem}
\begin{proof}
Let $x \circ y \circ z\subseteq \bigcap_{i \in \Delta}I_i$ for some $x,y,z \in R$ such that $x \circ z \nsubseteq  J(R)$ and $y \circ z \nsubseteq  J(R)$. This implies that  $x \circ y \circ z \subseteq I_i$ for every $i \in \Delta$. Since $I_i$ is a 2-absorbing $J$-prime hyperideal of $R$ for every $i \in \Delta$, we get the result that $x \circ y \subseteq  I_i$  and so $x \circ y \subseteq \bigcap_{i \in \Delta}I_i$. Thus $\bigcap_{i \in \Delta}I_i$ is a 2-absorbing $J$-prime hyperideal of $R$.
\end{proof}
\begin{definition}
A proper hyperideal $I$ of $R$ is called 2-absorbing $J$-primary if whenever elements $x,y,z \in R$ and $x \circ y \circ z \subseteq I$, then $x \circ y \subseteq I$ or $x \circ z \in J(I)$ or $y \circ z \in J(I)$.
\end{definition}
\begin{theorem}
Let $I$ be a hyperideal of $R$ . Then $I$ is a 2 -absorbing $J$-prime hyperideal of $R$ if and only if $I$ is a 2-absorbing $J$-primary hyperideal of $R$ with $J(I)=J(R)$.
\end{theorem}
\begin{proof}
$\Longrightarrow$ Let $I$ be a 2 -absorbing $J$-prime hyperideal of $R$. Suppose that $x,y,z \subseteq I$ for some $x,y,z \in R$. This implies that $x \circ y \subseteq I$ or $x \circ z \subseteq J(R)$ or $y \circ z \subseteq J(R)$ as $I$ is  a 2 -absorbing $J$-prime hyperideal of $R$. Since $J(R) \subseteq J(I)$, we have $x \circ y \subseteq I$ or $x \circ z \subseteq J(I)$ or $y \circ z \subseteq J(I)$. Thus $I$ is a 2-absorbing $J$-primary hyperideal of $R$. It is clear that $J(R) \subseteq J(I)$.  For the reverse inclusion,
we have $I \subseteq J(R)$ by Theorem \ref{62}. Thus $J(I) \subseteq J(R)$ and so   $J(I) =J(R)$. 

$\Longleftarrow$ Let $I$ be a 2-absorbing $J$-primary hyperideal of $R$ such that $J(I)=J(R)$. Suppose that $x,y,z \subseteq I$ for some $x,y,z \in R$. This implies that $x \circ y \subseteq I$ or $x \circ z \subseteq J(I)$ or $y \circ z \subseteq J(I)$ as $I$ is a 2-absorbing $J$-primary hyperideal of $R$. Since $J(I) =J(R)$, we get $x \circ y \subseteq I$ or $x \circ z \subseteq J(R)$ or $y \circ z \subseteq J(R)$. This means $I$ is a 2 -absorbing $J$-prime hyperideal of $R$.
\end{proof}
\begin{theorem} 
Let $(R_1,+_1,\circ_1)$ and $(R_2,+_2,\circ_2)$ be two multiplicative hyperrings and let $I_1$ and $I_2$ be hyperideals of $R_1$ and $R_2$, respectively. Then $I_1 \times I_2$ is a 2-absorbing $J$-prime hyperideal of $R_1 \times R_2$ if and only if $I_1$ is a $J$-prime hyperideal of $R_1$ and $I_2$ is a $J$-prime hyperideal of $R_2$.
\end{theorem}
\begin{proof}
$\Longrightarrow$ Let $I_1 \times I_2$ is a 2-absorbing $J$-prime hyperideal of $R_1 \times R_2$. Suppose that  $a \circ_1 b \subseteq I_1$ such that $b \notin J(R_1)$. Therefore $(a,1) \circ (1,0) \circ (b,1)=\{(x,y) \ \vert \ x \in a \circ_1 1 \circ_1 b , y \in 1 \circ_2 0 \circ_2 1\} \subseteq I_1 \times I_2$, $(a,1) \circ (b,1)=\{(x^\prime,y^\prime) \ \vert \ x^\prime \in a \circ_1 b, y^\prime \in 1 \circ_2 1\} \nsubseteq J(R_1 \times R_2)$ and  $(b,0) \in (1,0) \circ (b,1) \nsubseteq J(R_1 \times R_2)$. Hence we get $(a,0) \in (a,1) \circ (1,0) \subseteq I_1 \times I_2$ and so $a \in I_1$. Consequently, $I_1$ is a $J$-prime hyperideal of $R_1$. By a similar argument, we can prove that $I_2$ is a $J$-prime hyperideal of $R_2$.

$\Longleftarrow$ Let $I_1$ be a $J$-prime hyperideal of $R_1$ and $I_2$ is a $J$-prime hyperideal of $R_2$. Suppose that $(a_1,b_1) \circ (a_2,b_2) \circ (a_3,b_3) \subseteq I_1 \times I_2$ such that $(a_2,b_2) \circ (a_3,b_3) \nsubseteq J(R_1 \times R_2)$. Take  $a \in a_2 \circ_1 a_3$. Then $a_1 \circ_1 a \subseteq a_1 \circ_1 a_2 \circ_1 a_3 \subseteq I_1$. We assume that  $a_2 \circ_1 a_3 \nsubseteq J(R_1)$ then $a \notin J(R_1)$ as $J(R_1)$ is a {\bf C}-hyperideal.  Since  $I_1$ is a $J$-prime hyperideal of $R_1$, $a_1 \in I_1$ and so $a_1 \circ_1 a_2 \subseteq I_1$. On the other hand, if $b_1 \circ_2 b_3 \subseteq J(R_2)$, then $(a_1,b_1) \circ (a_3,b_3) \subseteq I_1 \times J(R_2) \subseteq J(R_1 \times R_2)$, as needed. Then we may assume $b_1 \circ b_3 \nsubseteq J(R_2)$. Take $b \in b_1 \circ_2 b_3$. Then $b_2 \circ_2 b \subseteq I_2$. Since $I_2$ is a $J$-prime hyperideal of $R_2$ and $b \notin J(R_2)$, we have $b_2 \in I_2$ and so $b_1 \circ_2 b_2 \subseteq I_2$. This $(a_1,b_1) \circ (a_2,b_2) \subseteq I_1 \times I_2$. This implies that $I_1 \times I_2$ is a 2-absorbing $J$-prime hyperideal of $R_1 \times R_2$.
\end{proof}
\begin{theorem}
Let $(R_1,+_1,\circ_1)$, $(R_2,+_2,\circ_2)$ and $(R_3,+_3,\circ_3)$ be three multiplicative hyperrings with non zero identity. Then, $R_1 \times R_2 \times R_3$ has no 2-absorbing $J$-prime hyperideals.
\end{theorem}
\begin{proof}
Let $I_1 \times I_2 \times I_3$ is a 2-absorbing $J$-prime hyperideal of $R_1 \times R_2 \times R_2$ for some hyperideals $I_1, I_2$ and $I_3$  of $R_1$, $R_2$ and $R_3$, respectively. Since $(0,0,0) \in (1,1,0) \circ (0,1,1)  \circ (1,0,1) \cap I_1 \times I_2 \times I_3$, then $(1,1,0) \circ (0,1,1)  \circ (1,0,1) \subseteq I_1 \times I_2 \times I_3$. Since $(1,1,0)  \circ (0,1,1) \nsubseteq J(R_1 \times R_2 \times R_3)$ and $(1,1,0)  \circ (1,0,1) \nsubseteq J(R_1 \times R_2 \times R_3)$, then $(0,0,1) \in (0,0,1) \circ (1,0,1) \subseteq I_1 \times I_2 \times I_3$. Moreover, we can get $(0,1,0), (1,0,0) \in I_1 \times I_2 \times I_3$. Thus $(1,1,1)=(1,0,0)+(0,1,0)+(0,0,1) \in I_1 \times I_2 \times I_3$ which is a contradiction.
\end{proof}
%%%%%%%%%%%%%%%%%%%%%%%%%%%%%
%%%%%%%%%%%%%%%%%%%%%%%%%%%%%%%%%%%%%%%%%%%%%%%%%
%%%%%%%%%%%%%%%%%%%%%%%%%%%%%%%

\end{document}